\documentclass[preprint,12pt,3p]{elsarticle}

\usepackage{amssymb}
\usepackage{amsthm}
\usepackage{amsmath}
\usepackage[hidelinks,colorlinks=false]{hyperref}
\newtheorem{theorem}{Theorem}
\newtheorem{lemma}{Lemma}

\newtheorem{corollary}{Corollary}

\newtheorem{definition}{Definition}

\begin{document}

\begin{frontmatter}

\title{New construction of graphs
with high chromatic number\\ and small clique}

\author[label1, label2]{Hamid Reza Daneshpajouh}
\address[label1]{School of Mathematics, Institute For Research In Fundamental Sciences, Niavaran Bldg, Niavaran Square, Tehran, Iran}
\address[label2]{Moscow Institute of Physics and Technology, Institutskiy per. 9, Dolgoprudny, Russia 141700
}


\ead{hr.daneshpajouh@ipm.ir, hr.daneshpajouh@phystech.edu}



\begin{abstract}
In this note, we introduce a new method for constructing graphs with high chromatic number and small clique. Indeed, via this method, we present a new proof for the well-known Kneser's conjecture.
\end{abstract}


\end{frontmatter}


\section{Introduction}
In this note, all graphs are finite, simple and undirected. The complete graph on $n$ vertices is denoted by $\mathcal{K}_n$. The number of graph vertices in the largest complete subgraph of $G$, denoted by $\omega (G)$ is called the clique number of $G$. The girth of a graph is the number of edges in its shortest cycle. A proper (vertex) coloring is an assignment of labels or colors to each vertex of a graph so that no edge connects two identically colored vertices. The smallest number of colors needed for the proper coloring of a graph $G$ is the chromatic number, $\chi (G)$. For a given graph $H$, a graph $G$ is called $H$-free if no induced subgraph of $G$ is isomorphic to $H$. In particular, a $\mathcal{K}_3$-free graph is called a triangle-free graph. 

It is obvious that $\chi (G)\geq\omega (G)$. The chromatic and clique numbers of a graph can be arbitrarily far apart. There are various constructions of triangle-free graphs with arbitrarily large chromatic number. Probably the best known is due to Mycielski~\cite{Myc}. In 1955, he created a construction that preserves the property of being triangle-free but increases the chromatic number. For more references, see also Blanche Descartes~\cite{descartes1947three}, John Kelly and Leroy Kelly~\cite{kelly1954paths}, and Alexander Zykov~\cite{zykov1949some}. Erd\"{o}s~\cite{Erdos} with a deeper insight showed the existence of graphs that have high girths and still have arbitrarily large chromatic number by probabilistic means. Indeed, exploring the relation between clique number and other properties of graphs such as chromatic number, maximum degree, etc., is still an active and fascinating research area within mathematics. There are still a lot of open questions waiting to be solved in this area, such as bounding the chromatic number of triangle-free graphs with fixed maximum degree~\cite{Kostochka}, or more generally Reed's conjecture~\cite{Reed}. For more problems see~\cite{OpenProblem}. The insights gained from this literature review indicate that we need a deeper understanding of these spaces. Thus, it is of interest to know new ways of constructing such spaces. 

In~\cite{Daneshpajouh}, compatibility graphs and $G$-Tucker lemma were introduced to obtain a new topological lower bound on the chromatic number of a special family of graphs. Indeed, via the bound, a new method for finding test graphs was proposed by the author. In this note, we give a new way, based on compatibility graphs and $G$-Tucker lemma, for constructing graphs with high chromatic number and small clique. Finally, as a corollary of this method, we give a new proof of the famous Kneser's conjecture.

The organization of the note is as follows. In Section $2$, we set up notation and terminology, and we repeat the relevant material from~\cite{Daneshpajouh} that will be needed throughout the note. Finally, in Section $3$, our main results are stated and proved.   
\section{Preliminary and Notations}
In this section we review definitions and results as required for the rest of the note. Here and subsequently, $G$ stands for a non-trivial finite group, and its identity element is denoted by $e$. A partially ordered set or poset is a set and a binary relation $\leq$ such that for all $a, b, c\in P$:
$a\leq a$ (reflexivity); $a\leq b$ and $b\leq c$ implies $a\leq c$ (transitivity); and $a\leq b$ and $b\leq a$ implies $a = b$ (anti-symmetry). A pair of elements $a, b$ of a partially order set are called comparable if $a\leq b$ or $b\leq a$. A subset of a poset in which each two elements are comparable is called a chain. A function $f : P\to Q$ between partially ordered sets is order-preserving or monotone, if for all $a$ and $b$ in $P$, $a\leq_{P} b$ implies $f(a)\leq_{Q} f(b)$. If $X$ is a set, then a group action of $G$ on $X$ is a function $G\times X\to X$ denoted $(g, x)\mapsto g.x$, such that $e.x = x$ and $(gh).x = g.(h.x)$ for all $g, h$ in $G$ and all $x$ in $X$. A $G$-poset is a poset together with a $G$-action on its elements that preserves the partial order, i.e, if $x<y$ then $g.x<g.y$.

To provide topological lower bounds on the chromatic numbers of graphs, Hom complex  was defined by Lov\'{a}sz. For a recent account of the theory, we refer the reader to~\cite{kozlov2007combinatorial}. We need the following version of this concept. 
 
\begin{definition}[Hom poset]
Let $F$ be a graph with vertex set $\{1, 2\cdots , n\}$. For a graph $H$, we define Hom poset $Hom_{p}(F, H)$ whose elements are given by all $n$-tuples $(A_1, \cdots, A_n)$ of non-empty subsets of $V(H)$, such that for any edge $(i, j)$ of $F$ we have $A_i\times A_j\subseteq E(H)$. The partial order is defined by $A=(A_1,\cdots, A_n)\leq B=(B_1,\cdots, B_n)$ if and only if $A_i\subseteq B_i$ for all $i\in\{1,\cdots , n\}$.
\end{definition}

Let $Z_r=\{e=\omega^0,\cdots, \omega^{r-1}\}$ be the cyclic group of order $r$.The cyclic group $Z_r$ acts on the poset $Hom_{p}(\mathcal{K}_r, H)$ naturally by cyclic shift. More precisely, for each $\omega^{i}\in Z_r$ and $(A_1,\cdots, A_r)\in Hom_{p}(\mathcal{K}_r, H)$ , define $\omega^{i}.(A_1,\cdots, A_r)=(A_{1+i(\text{mod} r)},\cdots , A_{r+i(\text{mod} r)})$.  

To find a new bound on the chromatic number of a special family of graphs, a combinatorial analog of the Borsuk-Ulam theorem for $G$-spaces, $G$-Tucker lemma, wsa introduced in~\cite{Daneshpajouh}. To recall the lemma, we need to make some definitions. Consider the $G$-poset $G\times \{1,\cdots , n+1\}$ with natural $G$-action, $h.(g, i)\to (hg, i)$, and the order defined by $(h, x) < (g, y)$ if $x < y$. Also, let ${\left(G\cup\{0\}\right)}^n\setminus\{(0,\cdots , 0)\}$ be the $G$-poset whose action is $g.(x_1, \cdots, x_n) = (g.x_1, \cdots, g.x_n)$, and the order relation is given by:
$$x=(x_1, \cdots, x_n)\leq y=(y_1, \cdots, y_n)\quad\text{if}\quad x_i=y_i\quad\text{whenever}\quad x_i\neq 0.$$

\begin{lemma}[$G$-Tucker's lemma]
Suppose that $n$ is a positive integers, $G$ is a finite group, and 
$$\lambda : {\left(G\cup\{0\}\right)
}^n\setminus\{(0,\cdots , 0)\}\to G\times\{1, \cdots, (n-1)\}$$
is a map such that $\lambda (g.x)=g.\lambda(x)$ for all $g\in G$ and all $x$ in ${\left(G\cup\{0\}\right)
}^n\setminus\{(0,\cdots , 0)\}$. Then there exist two sets $X\leq Y$ and $e\neq g\in G$ such that $\lambda (X) = g.\lambda (Y)$. Throughout this note, such a pair is called the bad pair.
\end{lemma}
It is worth noting that, there is a more general result for the case $G=Z_p$, the cyclic group of prime order $p$, see the $Z_p$-Tucker lemma of Ziegler~\cite{Ziegler}. 
Let us finish this section by recalling the definition of compatibility graph from~\cite{Daneshpajouh}.
\begin{definition}[Compatibility graph]
Let $P$ be a $G$-poset. The compatibility graph of $P$, denoted by $C_P$, has $P$ as vertex set, and two elements $x, y\in P$ are adjacent if there is an element $g\in G\setminus\{e\}$ such that $x$ and $g.y$ are comparable in $P$.
\end{definition}

\section{New  graphs
with high chromatic number and small clique}
In this section we state and discuss the main results of this note.
\begin{theorem}
For every graph $H$ and $r\geq 2$, the graph $C_{Hom(\mathcal{K}_{r}, G)}$ is $\mathcal{K}_{r+1}$-free.
\end{theorem}

\begin{proof}
For an $r$-tuple $A=(A_1, \cdots, A_r)$, define $|A|= |\cup_{i=1}^{r} A_i|$. Suppose that vertices $A_1=(A_{1,1},\cdots , A_{1,r}), \cdots, A_{k}=(A_{k,1},\cdots , A_{k,r})$ form a 
clique of size $k$ in $C_{Hom(\mathcal{K}_{r}, G)}$. Without loss of generality assume that $|A_{k}|\geq |A_{i}|$ for each $1\leq i\leq k-1$. Now, by the definition, there are $1\leq s_1,\cdots , s_{k-1}\leq r-1$ such that $\omega^{s_j}.A_{j}\subseteq A_{k}$, for each $1\leq j\leq k-1$. We claim that $s_i\neq s_j$ for each $i\neq j$. Suppose, contrary to our claim, that $s_i= s_j=a$ for some $1\leq i,j\leq k-1$. Since $A_i$ is connected to $A_j$ there is a, $1\leq b\leq r-1$ such that $\omega^{b}.A_i$ is comparable to $A_j$. Without loss of generality assume that $\omega^b.A_i\subseteq A_j$. Therefore
\[   
     \begin{cases}
       \omega^{a}.A_i\subseteq A_k\\
       \omega^{a}.A_j\subseteq A_k\\
       \omega^{b}.A_i\subseteq A_j\
     \end{cases}
\Longrightarrow
     \begin{cases}
       \omega^{a}.A_i\subseteq A_k \\
       \omega^{a+b}.A_i\subseteq A_k \
     \end{cases}
\Longrightarrow 
     \begin{cases}
       A_{i,a+b+1(\text{mod} r)}\subseteq A_{k, b+1} \\
       A_{i,a+b+1(\text{mod} r)}\subseteq A_{k, 1} \
     \end{cases}
\Longrightarrow A_{k,b+1}\cap A_{k,1}\neq\emptyset
\]
which contradicts the fact that any two distinct entries of $A_k$ are disjoint (note that, since $1\leq b\leq r-1$, $b+1\not\equiv 1 (\text{mod} r)$ which demonstrates that $A_{k,b+1(\text{mod} r)}$ and $A_{k, 1}$ are two distinct entries of $A_k$), therefore $k\leq r$. Now, the proof is completed.
\end{proof}

If $\psi : H\to K$ is a graph homomorphism, we associate it to a new map
$$C_{Hom(\mathcal{K}_r,-)}(\psi) : C_{Hom(\mathcal{K}_r, H)}\to C_{Hom(\mathcal{K}_r, K)}$$
by sending each $(A_1, \cdots, A_r)$ to $(\psi (A_1), \cdots, \psi (A_r))$. Since a complete bipartite subgraph of $H$ is sent by $\psi$ to a complete bipartite subgraph of $K$, the map $C_{Hom(\mathcal{K}_r,-)}(\psi)$ is a graph homomorphism. Moreover, the construction commutes with the composition of maps and the identity homomorphism is mapped to the identity homomorphism. So, using the terminology of category theory, one can say that $C_{Hom(\mathcal{K}_r,-)}(\psi)$ is a functor from the category of graphs to the category of $\mathcal{K}_{r+1}$-free graphs. Furthermore, we have an obvious graph homomorphism from $C_{Hom(\mathcal{K}_r, H)}$ to $H$ by sending each vertex $(A_1, \cdots, A_r)$ to $\min A_1$. This give us the following lower bound on chromatic number. 

\begin{corollary}
For any graph $H$, $\chi (C_{Hom(\mathcal{K}_r, H)})\leq\chi (H)$.
\end{corollary}

Recalling that the Kneser graph $KG(n,k)$ is the graph whose vertices correspond to the $k$-element subsets of the set $\{1,\cdots , n\}$, and two vertices are adjacent if and only if the two corresponding sets are disjoint. These graphs were introduced by Lov\'{a}sz~\cite{lovasz1978kneser} in his famous proof of the Kneser's conjecture~\cite{kneser}. In the context of graph theory, Kneser's conjecture is $\chi (KG(n,k)) = n-2k+2$, whenever $n\geq 2k-1$. Beside the Lov\'{a}sz's proof, there exist several different proofs for Kneser's conjecture, see for instance \cite{alishahi2015chromatic, barany1978shor, greene2002new, matouvsek2004combinatorial}. Now we are in a position to state and prove our main theorem.

\begin{theorem}
For all integers $n\geq rk$, $\chi\left(C_{Hom(\mathcal{K}_r, KG(n, k))}\right)\geq n-r(k-1)$. 
\end{theorem}

\begin{proof}
For any $X=(x_1,\cdots ,x_n)\in {\left(Z_r\cup\{0\}\right)}^{n}\setminus\{(0,\cdots , 0)\}$ and each $1\leq i\leq r$, define $X_i = \{s\,|\, x_s=\omega^{i}\}$. Also, denote by $\binom{X}{k}$ the $r$-tuple $(A_1,\cdots, A_r)$ where $A_i$ is the set of all $k$-subsets of $X_i$. Now, assume that $c : C_{Hom(\mathcal{K}_r, KG(n, k))}\to\{1, \cdots , C\}$ is a proper coloring of $C_{Hom(\mathcal{K}_r, KG(n, k))}$ with $C$ colors. We define a $\lambda $ coloring. Let $X\in {\left(Z_r\cup\{0\}\right)}^{n}\setminus\{(0,\cdots , 0)\}$. If $|X_i|\leq k-1$ for all $1\leq i\leq r$, set 
$$\lambda (X)= (\omega^{t}, \sum_{i=1}^{r}|X_i|),$$
where $\omega^t$ is the first nonzero element in $X$. if $|X_i|\geq k$ and $|X_j|\leq k-1$ for some $i, j$, set
$$\lambda (X)= (\omega^{t}, r(k-1) + |\{i : |X_i|\geq k\}|),$$
where $\omega^t$ is the first nonzero element in $X$ such that $|X_t|\geq k$.
Finally, consider the case that all of $|X_i|\geq k$. Let $\omega^t.X$ be the one of $X, \omega.X,\cdots , \omega^{r-1}.X$, with $c(\binom{\omega^{t}.X}{k}) = \min\{c\left(\binom{\omega^{i}.X}{k}\right) : 0\leq i\leq r-1\}$. Now, set
$$\lambda (X)= (\omega^{-t}, rk-1+ c\left(\binom{\omega^{t}.X}{k}\right)),$$
 
Note that, the vertices $\{\binom{\omega^{j}.X}{k}\, :\, 0\leq j\leq r-1\}$ form a clique of size $r$ in $C_{Hom(\mathcal{K}_r, KG(n, k))}$, $\lambda (X)\leq rk-1+C-r+1=C+r(k-1)$. Thus, $\lambda$ takes its values in $\{1, \cdots , C+r(k-1)\}$. It is easy to check that $\lambda(\omega^i.X)=\omega^i.\lambda(X)$ for all $X\in  {\left(Z_r\cup\{0\}\right)}^{n}\setminus\{(0,\cdots , 0)\}$ and all $\omega^i\in Z_r$. So, to use $Z_r$-Tucker lemma, its enough to show that $\lambda$ cannot have a bad pair. Let $X\leq Y$, $\lambda (X)=(\omega^{i}, a)$ and $\lambda (Y)=(\omega^{j}, a)$. We consider three cases.
\begin{itemize}
    \item The first case is $a\leq r(k-1)$. In this case, we have $X_i\subset Y_i$ for all $i$ and $\sum_{i=1}^{r}|X_i| = \sum_{i=1}^{r}|Y_i|$. These imply that $X=Y$. Thus, $\omega^i=\omega^j$.
    \item The second case is $r(k-1)+1\leq a \leq rk-1$. In this case, the number of $i$ for which  $|X_i|\geq k$ is equal to the number of $i$ for which $|Y_i|\geq k$. This fact beside the fact that $X_i\subseteq Y_i$ for all $i$, imply that $|X_i|\geq k$ if and only if $|Y_i|\geq k$. Therefore, $\omega^i=\omega^j$.
    \item Finally, assume that $a\geq rk$ and $\omega^i\neq\omega^j$. On the one hand, by the definition of $\lambda$, $c(\binom{\omega^{-i}.X}{k})=c(\binom{\omega^{-j}.Y}{k})$. On the other hand, since $X\leq Y$ and $\omega^i\neq\omega^j$, the vertex $\binom{\omega^i.X}{k}$ is connected to the vertex $\binom{\omega^i.Y}{k}$ which contradicts the proper coloring of $c$.
\end{itemize}
Applying $Z_r$-Tucker's lemma we have $C+ r(k-1)\geq n$, thus $C\geq n-r(k-1)$.
\end{proof}

In fact, the existence of graphs with high chromatic number and small clique is a direct consequence of Theorem $3.1$ and Theorem $3.2$. As another application, one could deduce a new proof of Kneser's conjecture. More precisely, it is well-known that the Kneser graph $KG(n,k)$ has a coloring with $n-2k+2$ colors (see~\cite[Section3.3]{matousek2008using}), thus by Theorem $3.2$ and Corollary $3.1$ we have
$$n-2k+2\leq\chi\left( C_{Hom(\mathcal{K}_2, KG_{n, k})}\right)\leq\chi\left(KG(n,k)\right)\leq n-2k+2.$$
In summary we have the following corollary. 
\begin{corollary}
For all $k\geq 1$ and $n\geq 2k-1$, the chromatic number of $C_{Hom(\mathcal{K}_2, KG_{n, k})}$ is $n-2k+2$. In particular, $\chi (KG(n, k))=n-2k+2$.
\end{corollary}

\section*{Acknowledgements}
This is part of the author's Ph.D. thesis, under the supervision of Professor Hossein Hajiabolhassan. I would like to express my sincere gratitude to him for the continuous support of my Ph.D. study and related research, for his patience, motivation, and immense knowledge.





\end{document}